\documentclass[a4paper,10pt]{article}%
\usepackage{eurosym}
\usepackage{amsmath}
\usepackage{mathrsfs}
\usepackage{amsfonts}
\usepackage{graphicx}
\usepackage{amsfonts}
\usepackage{amssymb}%
\usepackage{amsthm}
\usepackage{titling}
\usepackage{esint}
\usepackage{color}
\usepackage{float}
\usepackage{authblk}
\usepackage{enumitem}
\usepackage{amsmath}
\usepackage{graphicx}
\usepackage{tikz}
 \usetikzlibrary{positioning}
\usepackage{tikz-cd}

\usepackage{amssymb} 
\usepackage[abbrev]{amsrefs}

\usepackage{palatino}

\usepackage{amssymb}

\providecommand{\U}[1]{\protect\rule{.1in}{.1in}}
\newtheorem{theorem}{Theorem}[section]

\newtheorem{lemma}[theorem]{Lemma}

\theoremstyle{definition}

\newtheorem{remark}[theorem]{Remark}

\numberwithin{equation}{section}

\DeclareMathOperator\supp{supp}

\title{A self-improving property of Riesz potentials in $BMO$}

\author[Y.-W. Chen]{You-Wei Benson Chen}

\begin{document}

\maketitle

\begin{abstract}
    In this paper we prove that for non-negative measurable functions $f$, 
    \begin{align*}
I_\alpha f \in BMO(\mathbb{R}^n) \text{ if and only if } I_\alpha f \in BMO^\beta(\mathbb{R}^n) \text{ for } \beta \in (n-\alpha,n].
    \end{align*}
Here $I_\alpha$ denotes the Riesz potential of order $\alpha$ and $BMO^\beta$ represents the space of functions of bounded $\beta$-dimensional mean oscillation introduced in \cite{Chen-Spector}.  

\end{abstract}

\section{Introduction}


The space of functions of bounded $\beta$-dimensional mean oscillation, denoted as $BMO^\beta(\mathbb{R}^n)$, was introduced in \cite{Chen-Spector} recently. This space extends the classical notion of bounded mean oscillation to include the Choquet integral with respect to Hausdorff content. Specifically, for $0 < \beta \leq  n \in \mathbb{N}$, a function $u$ is said to be in $BMO^\beta(\mathbb{R}^n)$ if it is locally integrable with respect to the $\beta$-dimensional Hausdorff content $\mathcal{H}^\beta_\infty$ (see Section 2 for definition) and satisfies the condition

\begin{align*}
  \sup_{Q \subseteq \mathbb{R}^n}  \inf_{c \in \mathbb{R}}  \frac{1}{\ell(Q)^\beta} \int_Q | u(x)-c| \; d \mathcal{H}^\beta_\infty < \infty .
\end{align*}
  In here, $\ell(Q)$ is the length of the side of the cube and $\mathcal{H}^\beta_\infty$ is the Hausdroff content defined as
\[\mathcal{H}_{\infty}^{\beta}(E):=\inf\left\{\sum_{i=1}^\infty \omega_{\beta} r_i^{\beta}:E\subset \bigcup_{i=1}^\infty B(x_i,r_i)\right\},\]
where $\omega_{\beta} := {\pi^{\beta/2}}/{\Gamma\left(\frac{\beta}{2}+1\right)}$ is a normalized constant.

An analogue of the John-Nirenberg inequality for $BMO^\beta$ is also given in \cite[Theorem 1.3]{Chen-Spector}, which asserts that the functions of $\beta$-dimensional mean oscillation also have exponential integrability with respect to Hausdorff content. That is, if $u \in BMO^\beta(\mathbb{R}^n)$, then there exist constants $c_1= c_1(\beta),$ $c_2= c_2(\beta)>0$ such that for any cube $Q$ and any $\lambda>0$,
\begin{align}\label{jninequality_content_prime}
\mathcal{H}^{\beta}_\infty\left(\{x\in Q:|u(x)-c_Q|> \lambda\}\right) \leq c_1 l(Q)^\beta \exp(-c_2  \lambda /\|u \|_{BMO^\beta(\mathbb{R}^n)}),
\end{align} 
where $c_Q \in \mathbb{R}$ is a suitable constant depends on $Q$.  In particular, by integrating respect to $\lambda$ in (\ref{jninequality_content_prime}) and utilizing the pointwise estimate $t \leq \exp{ (t) }$ for $0\leq t$, one obtains the equivalence relation: $u  \in BMO^\beta (\mathbb{R}^n)$ if and only if there exist constants $\gamma ,C' >0$ independent of $Q$ such that
\begin{align}\label{betaexpintegral}
    \int_Q \exp{ \left( \gamma |u- c_Q| \right)} \; d\mathcal{H}^\beta_\infty \leq C' \ell(Q)^\beta
\end{align}
for every finite subcube $Q \subseteq \mathbb{R}^n$ and suitable constant $c_Q \in \mathbb{R}$ depending on $Q$. Moreover, it is a consequence of  (\ref{jninequality_content_prime}) that for $0 < \beta_1 \leq \beta_2 \leq n$, there exists a constant $C = C(\beta_1, \beta_2)$ such that
\begin{align}\label{inclusionBMObeta}
\Vert u \Vert_{BMO^{\beta_2} (\mathbb{R}^n)} \leq C \Vert u \Vert_{BMO^{\beta_1} (\mathbb{R}^n)}
\end{align}
(see \cite[Corollary 1.6]{Chen-Spector}).

The aim of this paper is to explore the mapping properties of Riesz potentials $I_\alpha$ into $BMO^\beta(\mathbb{R}^n)$ spaces, focusing specifically on the Morrey spaces $\mathcal{M}^\alpha_p (\mathbb{R}^n)$ and weak Lebesgue spaces $L^{n/\alpha,\infty} (\mathbb{R}^n)$. Our motivation is inspired by exponential integrability concerning Radon measures satisfying ball growth conditions. This includes several inequalities analogous to (\ref{jninequality_content_prime}) have been established (see, for example, \cite[Theorem 3]{Adams_1972}, \cite[p.~210]{Adamsbook} , \cite[Corollary 1.4.1 and Corollary 8.6.2]{Mazja_1985}, \cite[Theorem 2.2 and Theorem 2.5]{Cianchi_2008}, and especially the results of Adams \cite{Adams_1975}, Adams and Xiao \cite{Adams_2015} (see also \cite{Adams_2011}), Mart\'{\i}nez and Spector \cite[Corollary 1.6]{MD_2021}. To state their result, we recall that for $0<\alpha < n$ and $1\leq p \leq \frac{n}{\alpha}$, the Morrey space $\mathcal{M}^\alpha_p (\mathbb{R}^n)$ denotes the collection of all the measurable functions $f$ in $\mathbb{R}^n$ satisfying
\begin{align}\label{defMorrey}
    \Vert f \Vert_{\mathcal{M}^\alpha_p (\mathbb{R}^n)}:= \sup\limits_{Q\subseteq \mathbb{R}^n} \ell(Q)^{\alpha } \left( \frac{1}{\ell(Q)^n}  \int_Q |f|^p \; dy   \right)^{\frac{1}{p}} < \infty.
\end{align}
Moreover, a locally finite Radon measures $\mu$ is said to be belong to the Morrey space $\mathcal{M}^\beta (\mathbb{R}^n)$ if
\begin{align*}
\sup_{x \in \mathbb{R}^n,r>0} \frac{|\mu (B(x,r))|}{\;r^{\beta}} <\infty,
\end{align*}
where $|\mu|$ denotes the total variation of $\mu$.  As the topological dual of $L^1(\mathcal{H}^{\beta}_\infty; \mathbb{R}^n)$ can be identified with the Morrey space $\mathcal{M}^\beta (\mathbb{R}^n)$(see (\ref{HB})), the result of Adams and Xiao \cite[Theorem 0.1 (iii)]{Adams_2015} can be stated as 

\begin{theorem}[Adams and Xiao \cite{Adams_2015}]\label{Adams2015theorem}
    Let $0<\alpha < n$, $1\leq p \leq \frac{n}{\alpha}$ and $f$ be a non-negative function with support in a bounded domain $\Omega \subseteq \mathbb{R}^n$. If $f \in \mathcal{M}^\alpha_p (\mathbb{R}^n)$, then there exist constants $\gamma$, $C'' >0$ depend on $\alpha$, $\beta$, $n$, $p$ and $\Omega$ such that 
    \begin{align}\label{expAdams2015}
        \int_\Omega \exp{ \left( \frac{\gamma | I_\alpha f|}{ \Vert f \Vert_{\mathcal{M}^\alpha_p(\mathbb{R}^n)} } \right) } \; d \mathcal{H}^\beta_\infty < C'',
    \end{align}
holds  for $\beta \in ( n- \alpha p ,n]$.
\end{theorem}
Here $I_\alpha f $ denotes the Riesz potential of order $\alpha $  of the function $f$, which is defined as
\begin{align*}
I_\alpha f(x):= \frac{1}{\gamma(\alpha)} \int_{\mathbb{R}^n} \frac{f(y)}{|x-y|^{n-\alpha}} \; dy
\end{align*}
where $\gamma(\alpha)$ is a normalization constant (see \cite[p.~117]{S}). An interesting connection to earlier findings by Adams \cite[Corollary (i)]{Adams_1975} emerges when $p=1$ in Theorem \ref{Adams2015theorem}. It gives a necessary and sufficient condition for the Riesz potential of a non-negative function lies in the $BMO(\mathbb{R}^n)$ space:

\begin{theorem}[Adams Corollary (i) \cite{Adams_1975}]\label{Adamsequivalent}
Let $f$ be a non-negative measurable function.  Then
\begin{align}\label{Adamsequivalence1975}
I_\alpha f \in BMO(\mathbb{R}^n) \text{ if and only if } f \in \mathcal{M}_1^\alpha(\mathbb{R}^n) \text{ and } I_\alpha f \in L^1_{\text{loc}}(\mathbb{R}^n),
\end{align}
and in this case we have
\begin{align*}
    \Vert I_\alpha f \Vert_{BMO(\mathbb{R}^n)} \cong  \Vert f \Vert_{\mathcal{M}^\alpha_1 (\mathbb{R}^n)}.
\end{align*}

\end{theorem}
\begin{remark}
In here, we note that the assumption $I_\alpha f \in L^1_{\text{loc}} (\mathbb{R}^n)$ on the left-hand side of (\ref{Adamsequivalence1975}) is omitted in the statement of \cite[Corollary (i)]{Adams_1975}. Yet there exists a function $f \in L^{n/\alpha} (\mathbb{R}^n) \subseteq \mathcal{M}^\alpha_1 (\mathbb{R}^n)$ for which $I_\alpha f = \infty$ everywhere in $\mathbb{R}^n$ (see \cite[p.~432]{MR3381284} for example), one sees that such an assumption is needed. In fact, the proof of \cite[Proposition 3.4]{Adams_1975} regarding sufficiency is unaffected, and the assumption $I_\alpha f \in L^1_{\text{loc}}(\mathbb{R}^n)$ is introduced in \cite[Proposition 3.2]{Adams_1975} to establish the necessary condition for (\ref{Adamsequivalence1975}).
\end{remark}
The first result in this paper is a futher improvement of exponetial integrability in Theorem \ref{Adams2015theorem}, which shows the estimate (\ref{expAdams2015}) can actually be derived from John-Nirenberg inequality (\ref{betaexpintegral}) for  bounded $\beta$-dimensional mean oscillation functions.

\begin{theorem}\label{MainMorreyTheorem}
  Let $0<\alpha < n$ and $1\leq p < \frac{n}{\alpha}$.  If $f \in \mathcal{M}^{\alpha}_p(\mathbb{R}^n)$ and $I_\alpha f \in L^1_{loc}(\mathbb{R}^n)$, then for $ \beta \in (n- \alpha p  , \;n]$, there exists a constant $C>0$ depending on $n$, $\alpha$, $\beta$, $p$ such that
  \begin{align*}
      \Vert I_\alpha f \Vert_{BMO^{\beta}(\mathbb{R}^n)} \leq C \Vert f \Vert_{\mathcal{M}^\alpha_p (\mathbb{R}^n)}. 
  \end{align*}
In particular, there exist constants $\gamma , C_0 $ independent of $f$ and $c_Q \in \mathbb{R}$ depending on $Q$ such that
\begin{align*}
    \int_Q \exp{ \left(   \frac{\gamma \left |I_\alpha f -c_Q \right| }{ \Vert f \Vert_{\mathcal{M}^\alpha_p (\mathbb{R}^n)}}\right) \; d \mathcal{H}^\beta_\infty } \leq C_0 \ell(Q)^\beta 
\end{align*}
holds for every cube $Q \subseteq \mathbb{R}^n$.
\end{theorem}
As $f \in \mathcal{M}^\alpha_p (\mathbb{R}^n)$ having compact support implies $I_\alpha f \in L^1_{loc } (\mathbb{R}^n)$ (see \cite[p.~436]{MR3381284} for example), one easily sees that Theorem \ref{MainMorreyTheorem} refines Theorem \ref{Adams2015theorem}.

Motivated by Theorem \ref{MainMorreyTheorem}, one may inquire into the specific function space conditions under which a function $f$ belongs to, leading to $I_\alpha f \in BMO^\beta(\mathbb{R}^n)$ for any $\beta \in (0, n]$.  Inspired by \cite[Corollary 1.6]{MD_2021}, which asserts that for every open bounded set $\Omega \subseteq \mathbb{R}^n$ and $\beta \in (0 ,n]$, there exist constants $\gamma,\;C >0$ that depend on $n,$ $\alpha$, $\beta$ and $\Omega$ such that 
\begin{equation}
\int_\Omega \exp\left( \frac{\gamma |I_\alpha f|} { \Vert  f \Vert_{L^{n/ \alpha,\infty}(\Omega)}} \right) d\mathcal{H}^\beta_\infty \leq C,
\end{equation}
holds for $f \in L^{n / \alpha,\infty}(\Omega)$ with $\supp{f} \subseteq \Omega$, we introduce our second result. This result is a theorem analogous to Theorem \ref{MainMorreyTheorem} but concerns the scale of $f \in L^{n/ \alpha , \infty}(\mathbb{R}^n)$. The proof can be closely modeled after that of Theorem \ref{MainMorreyTheorem}. 

\begin{theorem}\label{potential_estimate}
Suppose that $f \in L^{n/\alpha,\infty}(\mathbb{R}^n)$ and $I_\alpha f \in L^1_{loc}(\mathbb{R}^n)$. Then there exists a constant $C= C(n,\alpha, \beta)$ such that
\begin{align*}
  \Vert  I_\alpha f  \Vert_{BMO^\beta(\mathbb{R}^n)} \leq C \Vert f \Vert_{L^{n/\alpha,\infty}(\mathbb{R}^n)}
\end{align*}
In particular, there exist constants $\gamma, C_0>0$, independent of $f$, and $c_Q \in \mathbb{R}$ depending on $Q$ such that
\begin{align*}
    \int_Q \exp{ \left(   \frac{\gamma \left |I_\alpha f -c_Q \right| }{ \Vert f \Vert_{L^{n/ \alpha ,\infty}(\mathbb{R}^n) }}\right) \; d \mathcal{H}^\beta_\infty } \leq C_0 \ell(Q)^\beta 
\end{align*}
holds for every cube $Q \subseteq \mathbb{R}^n$. 
\end{theorem}

We now introduce the main result of this paper, which establishes a new necessary and sufficient condition for non-negative functions $f$ such that $I_\alpha f$ is in $BMO(\mathbb{R}^n)$.

\begin{theorem}\label{Mainresult}
      Let $f$ be a non-negative measurable function. Then
\begin{align*}
    I_\alpha f \in BMO(\mathbb{R}^n) \text{ if  and only if }I_\alpha f \in BMO^\beta (\mathbb{R}^n) \text{ for  $\beta \in (n-\alpha,n]$}. 
\end{align*}
In this case, we have
\begin{align*}
    \Vert I_\alpha f \Vert_{BMO(\mathbb{R}^n)} \cong \Vert I_\alpha f \Vert_{BMO^\beta(\mathbb{R}^n)} \text{ for  $\beta \in (n-\alpha,n]$}. 
\end{align*}

\end{theorem}
This result builds on three key pieces: Adams' necessary and sufficient condition for $I_\alpha f \in BMO(\mathbb{R}^n)$ in Theorem \ref{Adamsequivalent},  the inclusion property (\ref{inclusionBMObeta}) from \cite{Chen-Spector}, and the embedding given in Theorem \ref{MainMorreyTheorem} when $p=1$.  The connections between these elements are illustrated in the diagram below:

\begin{tikzcd}[column sep=large, row sep=large, arrows=Rightarrow, scale=0.8, every node/.style={scale=0.8}]
& \fbox{$I_{\alpha} f \in BMO^{\beta} (\mathbb{R}^n) $ for $\beta \in (n-\alpha,n]$}  \arrow[dd, "(\ref{inclusionBMObeta})"', Rightarrow, shorten <=5pt, shorten >=5pt] \\
 \fbox{$f \in M^{\alpha}_1 (\mathbb{R}^n)$ and $I_{\alpha} f \in L_{loc}^{1}(\mathbb{R}^n)$} \arrow[ur, "Theorem \; \ref{MainMorreyTheorem}"', Rightarrow, shorten <=5pt, shorten >=5pt] \arrow[dr, "Theorem \; \ref{Adamsequivalent}", Leftrightarrow, shorten <=5pt, shorten >=5pt] & \\
& \fbox{$I_{\alpha} f \in BMO(\mathbb{R}^n)$}
\end{tikzcd}

Theorem \ref{Mainresult}, together with Theorem \ref{Adamsequivalent} and the characterization for $BMO^\beta(\mathbb{R}^n)$ discovered in \cite{Chen-Spector}, readily yields the following theorem:

\begin{theorem}\label{lotsofequivalent}
    Let $f$ be a non-negative measurable function. Then the following statements are equivalent.
\begin{enumerate}[label=(\roman*)]
\item  $I_\alpha f \in BMO(\mathbb{R}^n)$;
\item $I_\alpha f \in BMO^\beta(\mathbb{R}^n)$ for any $\beta \in (n-\alpha , n]$;
\item For any $\beta \in (n-\alpha , n]$, inequality (\ref{jninequality_content_prime}) holds with $u = I_\alpha f $;
\item For any $\beta \in (n-\alpha , n]$, inequality (\ref{betaexpintegral}) holds with $u = I_\alpha f $;
\item For any $\beta \in (n-\alpha , n]$, inequality
\begin{align*}
\|I_\alpha f \|_{BMO^{\beta,p}(\mathbb{R}^n)}:=  \sup_{Q \subset \mathbb{R}^n} \inf_{c \in \mathbb{R}} \left(\frac{1}{l(Q)^\beta} \int_{Q}  |I_\alpha f-c|^p d\mathcal{H}^{\beta}_\infty\right)^{1/p}< \infty
\end{align*}
holds for $p \in [1, \infty)$;
\item $f \in \mathcal{M}_1^\alpha (\mathbb{R}^n)$ and $I_\alpha f \in L^1_{loc}(\mathbb{R}^n)$ .
\end{enumerate}
\end{theorem}

This paper is organized as follows: Section 2 reviews the necessary preliminaries for defining the spaces $BMO^\beta(\mathbb{R}^n)$ and discusses their basic properties. This section also covers the mapping properties of the fractional maximal operator $\mathcal{M}_\eta$ which are important in the sequel. In Section 3, we establish a variant of technical results concerning the estimates of Riesz potentials, to be utilized in the subsequent sections. Finally, Section 4 presents the proofs of our main results.

\section{Preliminaries}

This section introduces the foundational concepts necessary for our discussion, focusing on the Choquet integral respect to the Hausdorff content, the $\beta$-dimensional $BMO$ space, and the fractional maximal function. Additionally, we outline their basic properties, which are crucial for the proof of our results.

In the following context, the notation $\Omega$ always denotes  either the entire Euclidean space $\mathbb{R}^n$ or an open cube $Q$ within $\mathbb{R}^n$. For a non-negative function $f$, the Choquet integral of $f$ over $\Omega$ with respect to the Hausdorff content, is defined by:

\begin{align}\label{choquet}
\int_{\Omega} f\;d\mathcal{H}^{\beta}_\infty:= \int_0^\infty \mathcal{H}^{\beta}_\infty\left(\left\{x \in \Omega: f(x)>t\right\}\right)\;dt.
\end{align}

A function $f: \Omega \rightarrow \mathbb{R}$ is said to be $\mathcal{H}^{\beta}_\infty$-quasicontinuous if, for any $\epsilon > 0$, there exists an open set $O \subseteq \Omega$ with $\mathcal{H}^{\beta}_\infty(O) < \epsilon$ where $f$ restricted to the complement of $O$ is continuous. This framework allows for the establishment of the $L^p$ norm for any signed $\mathcal{H}^{\beta}_\infty$-quasicontinuous function $f$, given $p \geq 1$, as:

\begin{align}\label{norm}
\|f\|_{L^p(\Omega;\mathcal{H}^{\beta}_\infty)}:=  \left( \int_{\Omega} |f|^p\;d\mathcal{H}^{\beta}_\infty \right)^{\frac{1}{p}}.
\end{align}
With these preparations, We introduce the vector space:
\begin{align*}
L^p(\Omega;\mathcal{H}^{\beta}_\infty):= \left\{ f\; \mathcal{H}^{\beta}_\infty\text{-quasicontinuous } : \|f\|_{L^p(\Omega;\mathcal{H}^{\beta}_\infty)}<+\infty\right\}.
\end{align*}

The norm defined in (\ref{norm}) behaves as a quasi-norm, and there exists an equivalent norm defined via Choquet integral respect to dyadic Hausdorff content such that $L^1(\Omega; \mathcal{H}^\beta_\infty)$ forms a Banach space. To be more precise, let $\mathcal{D}(Q)$ denote the set of all dyadic cube generalize by the cube $Q$ and we adopt the notation $\mathcal{D}(\mathbb{R}^n) = \mathcal{D}(\prod_{i=1}^n [0,1))$, the dyadic Hausdorff content $\mathcal{H}^{\beta,\Omega}_\infty$ subordinate to $\Omega$ of $E \subseteq \mathbb{R}^n$ is defined as
\begin{align}\label{defdyadic}
    \mathcal{H}^{\beta,\Omega}_\infty(E) : = \inf\left\{\sum_{i=1}^\infty \ell(Q_i)^\beta :E\subset \bigcup_{i=1}^\infty Q_i \text{ and } Q_i \in \mathcal{D}(\Omega)\right\}.
\end{align}
It has been shown that for non-negative functions $f$, there exists a constant $C_\beta$ such that 
\begin{align}\label{equivalentnorm}
    \frac{1}{C_\beta} \int_\Omega f \; d \mathcal{H}^{\beta,\Omega}_\infty \leq  \int_\Omega f \; d \mathcal{H}^{\beta}_\infty \leq C_\beta  \int_\Omega f \; d \mathcal{H}^{\beta,\Omega}_\infty. 
\end{align}
In here, $ \int_\Omega f \; d \mathcal{H}^{\beta,\Omega}_\infty $ denotes the Choquet integral of $f$ over $\Omega$ with respect to dyadic Hausdorff content $ \mathcal{H}^{\beta,\Omega}_\infty $, defined as 
\begin{align*}
    \int_{\Omega} f\;d\mathcal{H}^{\beta,\Omega}_\infty:= \int_0^\infty \mathcal{H}^{\beta,\Omega}_\infty \left(\left\{x \in \Omega: f(x)>t\right\}\right)\;dt.
\end{align*}
Moreover, the vector space $L^1(\Omega; \mathcal{H}^{\beta}_\infty)$ equipped with norm 
\begin{align}\label{dyadicnorm}
    \Vert f \Vert_{L^1 (\Omega, \mathcal{H}^{\beta,\Omega}_\infty)} : = \int_\Omega |f| \; d\mathcal{H}^{\beta,\Omega}_\infty
\end{align}
forms a Banach space (see \cite[Proposition 2.3]{YY_2008}, \cite[Proposition 3.2 and 3.5]{STW_2016} and the survey presented in \cite[Section 2]{Chen-Spector}).

The class of continuous and bounded functions in $\Omega$, denoted by $C_b(\Omega)$, for which the norm in equation (\ref{norm}) remains finite, is proven to be dense in the vector space $L^1(\Omega; \mathcal{H}^{\beta}_\infty)$ with quasi-norm $\Vert \cdot \Vert_{L^1 (\Omega; \mathcal{H}^{\beta}_\infty)}$ and thus also dense in the vector space $L^1(\Omega; \mathcal{H}^{\beta}_\infty)$ with norm $\Vert \cdot \Vert_{L^1 (\Omega; \mathcal{H}^{\beta,\Omega}_\infty)}$ (see \cite{AdamsChoquetnote} and \cite{PS_2023}). In addition, a function $f: \Omega \rightarrow \mathbb{R}^n$ is said to be locally integrable with respect to the Hausdorff content $\mathcal{H}^\beta_\infty$ in $\Omega$, denoted as $f \in L^1_{\text{loc}}(\Omega; \mathcal{H}^\beta_\infty)$, if $f$ multiplied by the characteristic function of any cube $Q \subseteq \Omega$ belongs to $L^1(Q; \mathcal{H}^\beta_\infty)$.

Next, we outline the basic properties of the Choquet integral with respect to the Hausdorff content. The following lemma can be found in \cite[p.~5]{Petteri_2023}, and its proof is referred to  \cite{AdamsChoquet1} and \cite[Chapter 4]{AdamsMorreySpacebook}.
\begin{lemma}\label{basicChoquet}
\begin{enumerate}[label=(\roman*)]
\item 
For $0 \leq a $ and non-negative functions $f$, we have
\begin{align*}
    \int_\Omega a \; f(x)\; d \mathcal{H}^\beta_\infty = a  \int_\Omega  \; f(x)\; d \mathcal{H}^\beta_\infty;
\end{align*}
\item For non-negative functions $f_1 $ and $f_2$, we have
\begin{align*}
    \int_\Omega  \; f_1(x) + f_2 (x)\; d \mathcal{H}^\beta_\infty \leq 2 \left( \int_\Omega  \; f_1(x)\; d \mathcal{H}^\beta_\infty +\int_\Omega  \; f_2(x)\; d \mathcal{H}^\beta_\infty \right);
\end{align*}
\item Let $1 \leq p, \; p' $ and $\frac{1}{p} + \frac{1}{p'} = 1$. Then for non-negative functions $f_1$ and $f_2$, we have
\begin{align*}
    \int_\Omega f_1 (x) f_2(x) \; d \mathcal{H}^\beta_\infty \leq 2 
    \left( \int_\Omega f_1(x)^p d \mathcal{H}^\beta_\infty  \right)^{\frac{1}{p}}  \left( \int_\Omega f_2(x)^{p'} d \mathcal{H}^\beta_\infty  \right)^{\frac{1}{p'}}.
\end{align*}
\end{enumerate}
\end{lemma}

It has been shown in \cite{AdamsChoquetnote} that the dual space of $L^1(\mathbb{R}^n; \mathcal{H}^\beta_\infty)$ is the Morrey space $\mathcal{M}^\beta(\mathbb{R}^n)$ and thus
\begin{align}\label{HB}
\int_{\mathbb{R}^n} f\;d\mathcal{H}^{\beta}_\infty \asymp \sup_{\|\mu\|_{\mathcal{M}^\beta} \leq 1} \int_{\mathbb{R}^n} f \;d\mu 
\end{align}
holds for every non-negative function $f \in L^1(\mathbb{R}^n;\mathcal{H}^{\beta}_\infty)$. In particular, (\ref{HB}) also holds if $f$ is a non-negative lower semi-continuous function (see Adams \cite{AdamsChoquetnote}*{p.~118}).

We now introduce the definition of the space of functions with bounded \(\beta\)-dimensional mean oscillation, denoted as $BMO^\beta(\Omega)$. The definition for the case where $\Omega$ is a cube has been provided in \cite{Chen-Spector}, and the definition for $\Omega = \mathbb{R}^n$ can be readily modified with those in \cite{Chen-Spector}.  For a function $u \in L_{loc}^1(\Omega;\mathcal{H}^{\beta}_\infty)$, $u$ is said to be has bounded $\beta$-dimensional mean oscillation in $\Omega$ provided
\begin{align}\label{nondyadic_norm}
 \sup_{Q \subseteq \Omega} \inf_{c \in \mathbb{R}} \frac{1}{l(Q)^\beta}  \int_{Q}  |u-c| \;d\mathcal{H}^{\beta}_\infty<+\infty,
\end{align}
where the supremum is taken over all finite subcubes $Q \subseteq \Omega$.  We then define the space of functions of bounded $\beta$-dimensional mean oscillation in $\Omega$:
\begin{align*}
BMO^{\beta}(\Omega):=  \left\{ u \in L^1(\Omega;\mathcal{H}^{\beta}_\infty) :  \|u\|_{BMO^{\beta}(\Omega)}<+\infty\right\},
\end{align*}
where
\begin{align*}
\|u\|_{BMO^{\beta}(\Omega)} :=\sup_{Q \subseteq\Omega} \inf_{c \in \mathbb{R}} \frac{1}{l(Q)^\beta}  \int_{\Omega}  |u-c| \;d\mathcal{H}^{\beta}_\infty.
\end{align*}

The following two lemmas presented below, and their associated proofs for the case when $\Omega$ is a cube in $\mathbb{R}^n$, are given in \cite[Corollary 1.5 and Corollary 1.6]{Chen-Spector} and are derived from the extended John-Nirenberg inequality (\ref{jninequality_content_prime}). The proofs for the lemmas when $\Omega = \mathbb{R}^n$ are trivially
adapted to the proofs of case when $\Omega$ is a cube.

\begin{lemma}\label{nesting}
Let $0<\beta_1 \leq \beta_2 \leq n$ and suppose $u \in BMO^{\beta_1} (\Omega)$.  Then $u \in BMO^{\beta_2}(\Omega)$ and
\begin{align*}
\|u\|_{BMO^{\beta_2}(\Omega)} \leq C\|u\|_{BMO^{\beta_1} (\Omega)}
\end{align*}
for some constant $C=C(\alpha,\beta)>0$ independent of $u$.
\end{lemma}

\begin{lemma} \label{pseminorms}
Let $\beta \in (0,n]$.  There exists a constant $C=C(\beta)>0$ such that
\begin{align*}
\frac{1}{C}\|u\|_{BMO^{\beta}(\Omega)} \leq \|u\|_{BMO^{\beta,p}(\Omega)} \leq C p\|u\|_{BMO^{\beta}(\Omega)}
\end{align*}
for all $u \in BMO^\beta(\Omega)$.
\end{lemma}

We now introduce the definition of the fractional maximal operator. Let \(0 < \eta < n\), where \(n \in \mathbb{N}\). The fractional maximal operator, associated with a locally integrable function $f$ of order \(\eta\), is defined as

\begin{align*}
    \mathcal{M}_\eta f(x) := \sup_{r>0} \frac{1}{r^{n-\eta}}  \int_{B(x,r)} |f(y)| \; dy.
\end{align*}
In particular, a locally integrable function $f$  belongs to the Morrey space $\mathcal{M}^{\alpha}_1 (\mathbb{R}^n)$ if and only if
\begin{align*}
\sup\limits_{x \in \mathbb{R}^n} \mathcal{M}_{\alpha} f  (x)< \infty.
\end{align*}
Furthermore, the following estimates of weak and strong types hold:
\begin{lemma}\label{Adamslazylemma}
Let $0 < \eta <n.$

\begin{enumerate}[label=(\roman*)]
\item  If $1<p < \frac{n}{\eta}$ and $\beta=n- \eta p$, then there exists a constant $A_1$ depending on $n$, $\eta$ and $p$ such that
\begin{align*}
    \left( \int_{\mathbb{R}^n} (\mathcal{M}_\eta f )^p \; d \mathcal{H}^\beta_\infty \right)^{\frac{1}{p}} \leq A_1  \left( \int_{\mathbb{R}^n} |f|^p \; dx \right)^{\frac{1}{p}};
\end{align*}

\item If $p = \frac{\beta}{n- \eta}$ with $n-\eta \leq \beta \leq n$, then there exists a constant $A_2$ depending on $n$, $\eta$ and $p$ such that
\begin{align*}
    \Vert \mathcal{M}_\eta f \Vert_{L^{p,\infty} ( \mathbb{R}^n;\mathcal{H}^\beta_\infty )} \leq A_2  \int_{\mathbb{R}^n} |f| \; dx.
\end{align*}
In here, $L^{p,\infty}(\mathbb{R}^n; \mathcal{H}^\beta_\infty)$ denotes the vector spaces of all the functions $f: \mathbb{R}^n \to \mathbb{R}$ satisfying
\begin{align*}
   \Vert f \Vert_{L^{p,\infty} (\mathbb{R}^n; \mathcal{H}^\beta_\infty)}: = \sup_{\lambda> 0 } \lambda^{\frac{1}{p}} \mathcal{H}^\beta_\infty\left( \left\{ x\in \mathbb{R}^n: |f(x)| > \lambda \right\} \right)< \infty.
\end{align*}
\end{enumerate}
\end{lemma}

The proof of (i) follows from the one weight inequality in \cite[Theorem B]{Sawyer_1982}.  The proof of (ii), when \(\beta = n - \eta\), can be found in \cite[Lemma 3.2]{BZ_1974}. Furthermore, both (i) and (ii) are given in \cite[Theorem 7]{AdamsChoquet1}, where they are presented in terms of the Choquet integral concerning Hausdorff content, as in Lemma \ref{Adamslazylemma} (see also \cite{hatano2023choquet} for the results concerning the weak Choquet spaces).

\section{Auxiliary Results
}
The fist lemma in this section is an improvement of the result in  \cite[p.~439]{MR3381284} as the Morrey space $\mathcal{M}^{\alpha }_1  (\mathbb{R}^n)$ is a suitable subspace of $L^{n/\alpha}(\mathbb{R}^n)$.
\begin{lemma}\label{locallyintegrablaformodifiedRieszpotential}
   Let $0< N$, $0<\alpha <n \in \mathbb{N}$, $f \in  \mathcal{M}^{\alpha}_1 (\mathbb{R}^n)$ and $ \Tilde{I_\alpha^N} f$ be defined as
    \begin{align*}
        \Tilde{I_\alpha^N} f (x) : = \frac{1}{\gamma(\alpha)} \int_{\mathbb{R}^n}  \left( \frac{1}{|x-y|^{n-\alpha}} - \frac{\chi_{B_N} (y)}{|y|^{n-\alpha }} \right) f(y) \; d  y,
    \end{align*}
    where $B_N$ denotes the complement of closed ball with radius $N$, that is, $B_N := \{ x\in \mathbb{R}^n : |x|>N \}.$ Then $\Tilde{I_\alpha^N} f \in L^1_{loc}(\mathbb{R}^n)$.
\end{lemma}
\begin{proof}
   It is sufficient to show that 
\begin{align}\label{claimforlocalintegrable}
    \int_{|x|<M} \Tilde{I_\alpha^N} f  (x) \; dx <\infty \text{ for every } N<M< \infty .
\end{align}
To this end, we split the integral into two pieces:
\begin{align*}
    \int_{|x|<M} \Tilde{I_\alpha^N} f (x) \; dx &\leq \frac{1}{\gamma(\alpha)}  \int_{|x|<M} \int_{|y| < 2M}\left| \frac{1}{|x-y|^{n-\alpha}} - \frac{\chi_{B_N} (y)}{|y|^{n-\alpha }} \right| |f(y)| \; dy \; dx\\
    & \leq \frac{1}{\gamma(\alpha)} \int_{|x|<M} \int_{|y| \geq 2M}\left| \frac{1}{|x-y|^{n-\alpha}} - \frac{\chi_{B_N} (y)}{|y|^{n-\alpha }} \right| |f(y)| \; dy \; dx \\
    &:= I + II.
\end{align*}
For $I$, we have $ \frac{\chi_{B_N} (y)}{|y|^{n-\alpha }} \leq N^{\alpha -n}$ for every $y \in \mathbb{R}^n$ and thus Tonelli's theorem gives
\begin{align*}
    I  & \leq  \frac{1}{\gamma(\alpha)}  \int_{|x|<M} \int_{|y| < 2M}\left| \frac{1}{|x-y|^{n-\alpha}} + N^{\alpha - n} \right|  |f(y)| \; dy \; dx\\
    & = \frac{1}{\gamma(\alpha)}  \int_{|y| < 2M} \int_{|x|<M}\left| \frac{1}{|x-y|^{n-\alpha}}  + N^{\alpha - n} \right|   \; dx\; |f(y)| \; d y.
\end{align*}
Since $ f$ is locally finite and 
\begin{align*}
    \int_{|x| < M} \left| \frac{1}{|x-y|^{n-\alpha }} + N^{\alpha -n}  \right| \; dx \leq C \left( M^\alpha +  N^{\alpha -n} M^n \right) < \infty
\end{align*}
holds for $|y|< 2M$, we obtain $I < \infty.$

For $II$, using mean value theorem, we deduce that for $|x| <M$ and $|y| >2M$, there exists a constant $C$ depending on $n$ and $\alpha $ such that
\begin{align*}
\left| \frac{1}{|x-y|^{n-\alpha }} - \frac{\chi_{B_N} (y)}{|y|^{n-\alpha }} \right| &= \left| \frac{1}{|x-y|^{n-\alpha }} - \frac{1}{|y|^{n-\alpha }} \right|\\
&\leq C \frac{|x|}{|y|^{n-\alpha +1}}\\
&\leq C \frac{M}{|y|^{n-\alpha +1}}
\end{align*}
  and again by Tonelli's theorem one deduces that
\begin{align}\label{part1forestimateII}
   \nonumber II &\leq C \left( \int_{|x| < M} M \; dx \right) \left( \int_{|y| \geq 2M} \frac{|f(y)|}{|y|^{n-\alpha +1}} \; dy \right)\\
& \leq C M^{n+1} \int_{|y| \geq 2M} \frac{|f(y)|}{|y|^{n-\alpha +1}} \; d y.
\end{align}
  Moreover, a dyadic splitting yields that
\begin{align}\label{part2forestimateII}
\nonumber \int_{|y| \geq 2M} \frac{|f(y)|}{|y|^{n-\alpha +1}} \; d  y &= \sum_{k=0}^\infty \int_{2^k M \leq  |y| < 2^{k+1} M  } \frac{|f(y)|}{|y|^{n-\alpha +1}} \; d y \\\nonumber
    &\leq   \sum_{k=0}^\infty \left( 2^k M \right)^{-(n-\alpha +1)} \int_{|y| < 2^{k+1}M} |f(y)| \; dy \\\nonumber
    & \leq  \sum_{k=0}^\infty \left( 2^k M \right)^{-(n-\alpha +1)}  \left( 2^{k+1}  M \right)^{n-\alpha }  \mathcal{M}_\alpha f(0) \\ 
    &\leq  \frac{C}{M}  \mathcal{M}_\alpha f(0).
\end{align}
Combining (\ref{part1forestimateII}) and (\ref{part2forestimateII}), we obtain $II < \infty$ and therefore the proof is complete.
\end{proof}
We next use an argument based on \cite[p.~ 443]{MR3381284} to give the following result.

\begin{lemma}\label{goodtermforlocalintegrabilityofRieszpotential}
    Let $0< \alpha < n$ and $ f \in \mathcal{M}^{\alpha }_1 (\mathbb{R}^n)$. Then $I_\alpha f \in L^1_{loc } (\mathbb{R}^n)$ if and only if 
    \begin{align}\label{usefultermforRieszpotential}
        \int_{|y| >N}  \frac{|f(y)|}{|y|^{n-\alpha }} \; d y < \infty \text{ for any } N>0.
    \end{align}
\end{lemma}
\begin{proof}
  Suppose that $I_\alpha f \in  L^1_{loc } (\mathbb{R}^n)$. It follows from Lemma \ref{locallyintegrablaformodifiedRieszpotential} that $\Tilde{I_\alpha^N } f \in L^1_{loc}(\mathbb{R}^n)$ for any $N>0$ and thus for every $N>0$ there exists $x_0 \in \mathbb{R}^n$ such that both $I_\alpha f (x_0) $ and $\Tilde{I_\alpha^N }  f (x_0) $ are finite. In particular, this implies that
\begin{align*}
   \frac{1}{\gamma(\alpha)}  \int_{|y| >N}  \frac{f(y)}{|y|^{n-\alpha }} \; d(y) =I_\alpha f (x_0) - \Tilde{I_\alpha ^N} f (x_0) < \infty,
\end{align*}
and therefore
 \begin{align*}
        \int_{|y| >N}  \frac{|f(y)|}{|y|^{n-\alpha }} \; d  y < \infty \text{ for any } N>0.
    \end{align*}

Now suppose that (\ref{usefultermforRieszpotential}) holds. Since  $\Tilde{I_\alpha^N } f \in L^1_{loc} (\mathbb{R}^n)$ by Lemma \ref{locallyintegrablaformodifiedRieszpotential} and the identity 
\begin{align*}
    I_\alpha f (x) &= \Tilde{I_\alpha ^N} f(x) +  \frac{1}{\gamma(\alpha)}\int_{|y| >N}  \frac{f(y)}{|y|^{n-\alpha }} \; dy \\
    &\leq \Tilde{I_\alpha ^N} f(x) +  \frac{1}{\gamma(\alpha)}\int_{|y| >N}  \frac{|f(y)|}{|y|^{n-\alpha }} \; dy
\end{align*}
  holds if $\Tilde{I_\alpha^N } f (x)$ is well-defined and finite (and thus holds for Lebesgue almost every $x$), we obtain $I_\alpha f  \in L^1_{loc}(\mathbb{R}^n)$.
\end{proof}

\begin{lemma}\label{quasicontinuouslowerimplyhigher}
    
Let $0<\beta_1 \leq \beta_2 \leq n$ and $Q \subseteq \mathbb{R}^n$ be a cube. If $f: Q \to \mathbb{R}$ is $\mathcal{H}^{\beta_1}_\infty$-quasicontinuous in $Q$, then $f$ is $\mathcal{H}^{\beta_2}_\infty$-quasicontinuous in $Q$.

\end{lemma}
\begin{proof}
  Without loss of generality, we may assume $\mathcal{H}^{\beta_2}_\infty (Q) >0.$  Suppose that $f$ is $\mathcal{H}^{\beta_1}_\infty$-quasicontinuous in $Q$. Then there exists a sequence of functions $\{ \phi_j  \}_j^\infty \subseteq C_b (Q)$ such that
    \begin{align*}
        \lim\limits_{j \to \infty} \int_Q |f - \phi_j| \; d \mathcal{H}^{\beta_1}_\infty = 0 .
    \end{align*}
Yet an application of Lemma 2.2 in \cite{chen2023capacitary} and H\"older inequality in Lemma \ref{basicChoquet}  yields the estimate
\begin{align*}
    \int_Q |f - \phi_j| \;d\mathcal{H}^{\beta_2}_\infty& \leq \omega_{\beta_1}  \left( \frac{1}{ \omega_{\beta_2}} \right)^{\frac{\beta_2}{\beta_1}} \left( \int_Q | f- \phi_j|^{\frac{\beta_1}{\beta_2}} \; d \mathcal{H}^{\beta_1}_\infty\right)^{\frac{\beta_2}{\beta_1}}\\
  & \leq \omega_{\beta_1}  \left( \frac{1}{ \omega_{\beta_2}} \right)^{\frac{\beta_2}{\beta_1}} \left( 2 \left( \int_Q |f - \phi_j| \; \mathcal{H}^{\beta_1}_\infty \right)^{\frac{\beta_2}{\beta_1}} \mathcal{H}^{\beta_1 }_\infty (Q)^{1- \frac{\beta_1}{\beta_2}} \right)^{\frac{\beta_2}{\beta_1 }} \\
    &\leq 2^{\frac{\beta_2}{\beta_1}} \omega_{\beta_1}  \left( \frac{1}{ \omega_{\beta_2}} \right)^{\frac{\beta_2}{\beta_1}} \mathcal{H}^{\beta_1}_\infty (Q)^{\frac{\beta_2}{\beta_1} -1} \int_Q |f - \phi_j| \; d \mathcal{H}^{\beta_1}_\infty
\end{align*}
which converges to $0$ as $j$ goes to $\infty$. This completes the proof.
\end{proof}
We next present a lemma concerning the Riesz potential of a function $\phi \in \mathcal{M}^\alpha_1 (\mathbb{R}^n)$ with compact support. This lemma will help us verify the quasicontinuity of $I_\alpha f$ with respect to the Hausdorff content in Lemma \ref{forcompactmorrey}.

\begin{lemma}\label{Morre1ypreciseestimate}
      Let $0< \alpha<n$, $0<\varepsilon \leq \alpha$, $Q$ be an open cube in $\mathbb{R}^n$ with centre $x_0$ and $\phi$  be a real-valued function in $\mathbb{R}^n$. If $\supp{\phi} \subseteq 2Q$, then there exist constants $C'$ and $C''$ depending on $n$, $\alpha$ and $\varepsilon$ such that
\begin{align}\label{Morrey1first}
       \int_Q | I_\alpha \phi | \; d \mu \leq C' \ell(Q)^{ \varepsilon} \Vert \phi \Vert_{L^1(\mathbb{R}^n)}\leq C'' \ell(Q)^{n-\alpha +\varepsilon} \mathcal{M}_\alpha \phi (x_0).
\end{align}
holds for every $\mu \in \mathcal{M}^{n-\alpha +\varepsilon} (\mathbb{R}^n)$ with $\Vert \mu \Vert_{\mathcal{M}^{n-\alpha + \varepsilon} (\mathbb{R}^n)} \leq 1  $.  
\end{lemma}
\begin{proof}
 Let $\mu \in \mathcal{M}^{n-\alpha +\varepsilon} (\mathbb{R}^n)$ satisfying $\Vert \mu \Vert_{\mathcal{M}^{n-\alpha + \varepsilon} (\mathbb{R}^n)} \leq 1  $.  We first note that an usual dyadic splitting gives that for every $\eta \in [0,\alpha)$,
\begin{align}\label{usualdyadicsplitting}
\nonumber |I_\alpha \phi (x)|& \leq \frac{1}{\gamma(\alpha)}
\int_{2Q} \frac{|\phi(y)|}{|x-y|^{n-\alpha}}  \;dy \\
 \nonumber &\leq \frac{1}{\gamma(\alpha)} \sum_{k=1}^\infty \frac{1}{(2^{-k+5} l(Q))^{n-\alpha}} \int_{2^{-k+5}Q^x \setminus 2^{-k+4}Q^x} |\phi(y)| \;dy \\
&\leq \frac{C_\eta }{\gamma(\alpha)} \ell(Q)^{\alpha-\eta} \mathcal{M}_\eta \phi (x),
\end{align}
where $Q^x$ denotes the cube with centre $x$ and radius $\ell(Q)$. Moreover, Lemma \ref{Adamslazylemma} (ii) yields that for any $0\leq  \varepsilon <\alpha$, there exists a constant $A_2$ depending only on $n$, $\alpha$, $\varepsilon$ and $\eta$ such that
\begin{align*}
    \Vert \mathcal{M}_\eta \phi \Vert_{L^{q,\infty} ( \mathbb{R}^n; \mathcal{H}^{n - \alpha + \varepsilon}_\infty)} \leq A_2 \Vert \phi  \Vert_{L^{1}  (\mathbb{R}^n)},
\end{align*}
where $q = \frac{n-\alpha + \varepsilon}{n - \eta}$. Now choosing $\eta = \alpha - \frac{\varepsilon}{2} $, we have $q >1$ and therefore by H\"older inequality in Lorentz scale (see Theorem 1.4.16 in \cite{Grafakos_classical} for example) we deduce that there exists $C= C(n,\alpha,\varepsilon)$ such that
\begin{align}\label{mainestimateforIalphaphifirst}
\nonumber \int_{Q} \mathcal{M}_\eta  \phi (x) \; d\mu (x) &\leq C  \Vert \chi_{2Q} \Vert_{L^{q',1} (\mu)} \Vert \mathcal{M}_\eta \phi \Vert_{L^{q, \infty} (\mu)}\\ \nonumber
&\leq C   \ell(Q)^{\varepsilon+ \eta - \alpha} \Vert \mathcal{M}_\eta \phi \Vert_{L^{q, \infty} (\mathbb{R}^n;\mathcal{H}^{n-\alpha +\varepsilon}_\infty)}\\ 
& \leq  C \ell(Q)^{\varepsilon + \eta - \alpha} \Vert \phi \Vert_{L^1 (\mathbb{R}^n)} ,
\end{align}
where the second inequality follows by 
\begin{align*}
    \Vert \chi_{2Q} \Vert_{L^{q',1} (\mu)} \leq C \ell(Q)^{\frac{n-\alpha + \varepsilon}{q'}}= C\ell(Q)^{\varepsilon + \eta -\alpha}.
\end{align*}
The combination of (\ref{usualdyadicsplitting}) and (\ref{mainestimateforIalphaphifirst}) then gives
\begin{align}\label{forsharp1}
\nonumber \int_Q  | I_\alpha \phi (x) | \; d\mu(x)  &\leq C \ell(Q)^{\alpha - \eta} \int_{Q} \mathcal{M}_\eta  \phi (x) \; d\mu (x) \\ \nonumber
& \leq  C \ell(Q)^{\varepsilon} \Vert \phi \Vert_{L^1 (\mathbb{R}^n)} \\ \nonumber
& \leq C \ell(Q)^{n-\alpha + \varepsilon} \mathcal{M}_{\alpha }  \phi (x_0).
\end{align}
This completes the proof. 
\end{proof}

In the following lemma, we demonstrate that for functions $f$ satisfying the conditions of Lemma \ref{Morre1ypreciseestimate}, $I_\alpha f$ must belong to $L^1(Q;\mathcal{H}^\beta_\infty)$. Consequently, an analogue of estimate (\ref{Morre1ypreciseestimate}), in terms of the Choquet integral with respect to Hausdorff content, is established.

\begin{lemma}\label{forcompactmorrey}
    Let $0< \alpha<n$, $0<\varepsilon \leq \alpha$, $Q$ be an open cube in $\mathbb{R}^n$ with centre $x_0$ and $f$ be a real-valued function in $\mathbb{R}^n$. If $\supp{f} \subseteq 2Q$, then $I_\alpha f \in L^1(Q;\mathcal{H}^{n-\alpha +\varepsilon}_\infty)$ and there exists a constant $C$ depending on $n$, $\alpha$, $\varepsilon$ such that
    \begin{align}
        \int_Q | I_\alpha f  (x )| \; d \mathcal{H}^{n-\alpha + \varepsilon}_\infty \leq C \ell(Q)^{n-\alpha + \varepsilon} \mathcal{M}_\alpha f (x_0).
    \end{align}
\end{lemma}

\begin{proof}
By Lemma \ref{Morre1ypreciseestimate}, we have
\begin{align}\label{Morreylemmafirst}
       \int_Q | I_\alpha \phi | \; d \mu \leq C' \ell(Q)^{ \varepsilon} \Vert \phi \Vert_{L^1(\mathbb{R}^n)}\leq C'' \ell(Q)^{n-\alpha +\varepsilon} \Vert \phi \Vert_{\mathcal{M}^{ \alpha}_1 (\mathbb{R}^n)}
\end{align}
holds for every $\phi \in C_b(2Q)$ and 
 $\mu \in \mathcal{M}^{n-\alpha +\varepsilon} (\mathbb{R}^n)$ with $\Vert \mu \Vert_{\mathcal{M}^{n-\alpha + \varepsilon} (\mathbb{R}^n)} \leq 1  $.  Since $C_b(2Q)$ is dense in $L^1(2Q; \mathbb{R}^n)$ and $f \in L^1(2Q; \mathbb{R}^n) $, there exists a sequence of functions $\{ \phi_i \}_{i=1}^\infty \subseteq C_b(2Q)$ such that $\supp{\phi_i} \subseteq 2Q$ and 
\begin{align*}
    \int_{2Q} | \phi_i (x) - f(x)| \; dx \leq \frac{1}{4^i} \text{ for each } i \in \mathbb{R}^n.
\end{align*}
As $|I_\alpha \phi_i - I_\alpha \phi_j|$ is continuous, we may use estimate (\ref{Morre1ypreciseestimate}) and deduce that for $j \geq i$, there exists a constant $C>0$ depending on $n$, $\alpha$, $\varepsilon$ such that 
\begin{align}\label{quasiconvergentforIalphamu}
    \int_Q | I_\alpha \phi_i(x) - I_\alpha \phi_j (x)| \; d \mathcal{H}^{n-\alpha +\varepsilon}_\infty(x) \leq  \frac{ C \ell(Q)^{ \varepsilon} }{4^i} .
\end{align}
The combination of (\ref{equivalentnorm}) and (\ref{quasiconvergentforIalphamu}) then yields that there exists a constant $C'>0$ depending on $n$, $\alpha$ and $\varepsilon$ such that
\begin{align}
    \int_Q | I_\alpha \phi_i(x) - I_\alpha \phi_j (x)| \; d \mathcal{H}^{n-\alpha +\varepsilon,Q}_\infty(x) \leq  \frac{ C' \ell(Q)^{ \varepsilon} }{4^i} ,
\end{align}
where $\mathcal{H}^{n-\alpha +\varepsilon,Q}_\infty$ denotes the dyadic Hausdorff content given in (\ref{defdyadic}). Since the vector space $L^1(Q; \mathcal{H}^{n-\alpha +\varepsilon}_\infty)$ equipped with the norm $\Vert \cdot \Vert_{L^1(Q; \mathcal{H}^{n-\alpha +\varepsilon,Q}_\infty)}$ forms a Banach space, there exists $g \in L^1 (Q;\mathcal{H}^{n-\alpha +\varepsilon}_\infty)$ such that $I_\alpha \phi_i$ converges to $g$ in $L^1 (Q;\mathcal{H}^{n-\alpha +\varepsilon}_\infty)$ and 
\begin{align}\label{phitog}
    \lim\limits_{ i \to \infty} I_\alpha  \phi_i (x) = g(x) \text{ for }  \mathcal{H}^{n-\alpha +\varepsilon}_\infty \text{ almost every } x \in Q.
\end{align}

We next claim
\begin{align}\label{phitoIalphaf}
    \lim\limits_{ i \to \infty} I_\alpha \phi_i (x) = I_\alpha f (x) \text{ for }  \mathcal{H}^{n-\alpha +\varepsilon}_\infty \text{ almost every } x \in Q,
\end{align}
and thus $I_\alpha f \in L^1(Q;\mathcal{H}^{n-\alpha +\varepsilon}_\infty)$. To this end, we observe that an argument similar to (\ref{mainestimateforIalphaphifirst}) gives that 
\begin{align*}
    \int_Q \mathcal{M}_\eta ( \phi_i-f) (x) \; d \mu (x) 
&\leq C\ell(Q)^{\varepsilon + \eta - \alpha }\Vert  \phi_i - f \Vert_{L^{1 } (\mathbb{R}^n)} .
\end{align*}
Moreover, as $\mathcal{M}_\eta ( \phi_i -f ) \chi_Q$ is a lower semi-continuous function, we obtain by (\ref{HB}) that 
\begin{align}
    \int_Q \mathcal{M}_\eta ( \phi_i -f  )(x) \;d \mathcal{H}^{n-\alpha +\varepsilon}_\infty (x) \leq C \ell(Q)^{\varepsilon + \eta - \alpha } \Vert  \phi_i -f  \Vert_{L^{1 } (\mathbb{R}^n)},
\end{align}
where $\eta = \alpha - \frac{\varepsilon}{2}$.  As Fatou's lemma holds for Choquet integrals with respect to Hausdorff  content (see (1.2) in \cite{PS_2023} for example), we deduce that
\begin{align*}
  0& \leq   \int_Q \liminf\limits_{i\to 0} \mathcal{M}_\eta (\phi_i - f) (x ) \; d \mathcal{H}^{n-\alpha + \varepsilon}_\infty (x )\\
  & \leq \liminf\limits_{i\to 0} \int_Q  \mathcal{M}_\eta  (\phi_i - f) (x ) \; d \mathcal{H}^{n-\alpha + \varepsilon}_\infty (x )\\ 
  &\leq \liminf\limits_{i\to 0} C   \ell(Q)^{\varepsilon + \eta - \alpha }  \Vert  \phi_i -f  \Vert_{L^{1 } (\mathbb{R}^n)}\\
    & = 0.
\end{align*}
As a consequence, we conclude that
\begin{align*}
     \liminf\limits_{i\to 0} \mathcal{M}_\eta  (\phi_i - f  ) (x ) = 0 \text{ for } \mathcal{H}^{n-\alpha + \varepsilon}_\infty \text{ almost every } x\in Q.
\end{align*}
Yet a standard dyadic splitting similar to (\ref{usualdyadicsplitting}) yields that
\begin{align*}
    | I_\alpha \phi_i (x )- I_\alpha f(x)| \leq C_\eta \ell(Q)^{\frac{\varepsilon}{2}} \mathcal{M}_\eta  (\phi_i - f) (x ) \text{ for every } x\in Q
\end{align*}
leading to
\begin{align}\label{IaphaIalphaphii}
     \liminf\limits_{i\to 0} | I_\alpha \phi_i (x) - I_\alpha f(x)| = 0 \text{ for } \mathcal{H}^{n-\alpha + \varepsilon}_\infty  \text{ almost every } x\in Q.
\end{align}
Combining (\ref{phitog}) and (\ref{IaphaIalphaphii}), we obtain that
\begin{align*}
    \lim_{i \to 0} I_\alpha \phi_i (x) = g(x) = I_\alpha f (x) \text{ for } \mathcal{H}^{n-\alpha + \varepsilon}_\infty  \text{ almost every } x\in Q,
\end{align*}
as claimed.

Finally, since $ \chi_Q I_\alpha |f|$ is a lower semi-continuous function, the combination of (\ref{HB}) and (\ref{forsharp1}) yields that there exists a constant $C$ depending on $n$, $\alpha $ and $\varepsilon$ such that
\begin{align}\label{forsharp2}
  \nonumber   \int_Q |I_\alpha f | \; d \mathcal{H}^{n-\alpha +\varepsilon}_\infty  &\leq \int_{\mathbb{R}^n}  \chi_Q I_\alpha |f|  \; d \mathcal{H}^{n-\alpha +\varepsilon}_\infty\\
   \nonumber     & \leq C \sup_{\|\mu\|_{\mathcal{M}^{n-\alpha +\varepsilon}} \leq 1} \int_{Q} I_\alpha |f|  \;d\mu \\
  &\leq C \ell(Q)^{n-\alpha + \varepsilon} \mathcal{M}_{ \alpha } f (x_0).
\end{align}
This completes the proof.


\end{proof}
The subsequent Lemma improves Lemma \ref{forcompactmorrey} in the case where $p\geq 1$.
\begin{lemma}\label{Morreyforplemma}
    Let $0<\alpha < n$, $1 \leq p  \leq \frac{n}{\alpha}$, $Q$ be an open cube and $f$ be a real-valued function in $\mathbb{R}^n$. If $\supp{f} \subseteq 2Q$, $f \in \mathcal{M}^\alpha_p (\mathbb{R}^n)$ and $\beta \in (n-\alpha p,n]$, then $I_\alpha f \in L^1 (Q; \mathcal{H}^\beta_\infty)$ and there exist constant $C'$ and $C''$ depending on $n$, $\alpha$, $\beta$, $p$ such that
\begin{align}\label{extensionforcompactmorrey}
    \int_Q |I_\alpha f| \; d\mathcal{H}^\beta_\infty \leq C' \ell(Q)^{\beta+\alpha -\frac{n}{p}} \Vert f \Vert_{ L^p(\mathbb{R}^n)} \leq C'' \ell(Q)^\beta \Vert f \Vert_{\mathcal{M}^\alpha_p (\mathbb{R}^n)}.
\end{align}
   
\end{lemma}

\begin{proof}
It is sufficient to prove the case $1<p \leq \frac{n}{\alpha}$, as Lemma \ref{forcompactmorrey} yields (\ref{extensionforcompactmorrey}) for $p=1$. To this end, we again use the usual dyadic splitting in (\ref{usualdyadicsplitting}) and obtain that
\begin{align*}
| I_\alpha f  (x) |
&\leq \frac{C_\eta }{\gamma(\alpha)} \ell(Q)^{\alpha-\eta} \mathcal{M}_\eta f (x)
\end{align*}
holds for any $\eta \in [0,\alpha)$. Fix $\beta \in (n-\alpha p,n]$ and choose $\eta $ so that $n- \eta p =\beta$. Lemma \ref{Adamslazylemma} (i) then yields that there exists a constant $A_2$ depending only on $n$, $\beta$ and $p$ such that

\begin{align}\label{Saywer'sestimate}
    \Vert \mathcal{M}_\eta f \Vert_{L^{p} (\mathbb{R}^n; \mathcal{H}^{\beta}_\infty)} \leq A_2 \Vert f \Vert_{L^{p}  (\mathbb{R}^n)}.
\end{align}
 Applying H\"older's inequality of the Hausdorff content (see Lemma \ref{basicChoquet} (iii)) and estimate (\ref{Saywer'sestimate}), we deduce that there exists a constant $C= C(n,\alpha, \beta, p)$ such that
\begin{align}\label{Morreypmainestimate}
\nonumber \int_Q |I_\alpha f| \; d \mathcal{H}^\beta_\infty &\leq C \ell(Q)^{\alpha - \eta} \int_{Q} \mathcal{M}_\eta  f \; d\mathcal{H}^\beta_\infty\\ \nonumber
&\leq C \ell(Q)^{\alpha - \eta} \Vert \chi_{Q} \Vert_{L^{p'} (\mathbb{R}^n;\mathcal{H}^\beta_\infty)} \Vert \mathcal{M}_\eta f \Vert_{L^{p} (\mathbb{R}^n;\mathcal{H}^\beta_\infty)}\\ \nonumber
&\leq C \ell(Q)^{\beta + \alpha - \frac{n}{p}} \Vert f \Vert_{L^{p } (\mathbb{R}^n)}\\
&\leq C \ell(Q)^\beta \Vert f \Vert_{\mathcal{M}^\alpha_p (\mathbb{R}^n)},
\end{align}
which verifies (\ref{extensionforcompactmorrey}).
 Moreover, since $C_c^\infty (\mathbb{R}^n)$ is dense in $L^p (\mathbb{R}^n)$, there exists a sequence of functions $\{ \phi_j \}_{j=1 }^\infty \subseteq C^\infty_c(\mathbb{R}^n)$ such that for each $j$,
\begin{align*}
    \lim\limits_{j \to \infty} \Vert f - \phi_j \Vert_{L^p (\mathbb{R}^n)} =0.
\end{align*} 
In particular, it follows by (\ref{Morreypmainestimate}) that
\begin{align*}
     \int_Q |I_\alpha f - I_\alpha \phi_j| \; d \mathcal{H}^\beta_\infty 
     & \leq C \ell(Q)^{\alpha + \beta - \frac{n}{p}} \Vert f - \phi_j \Vert_{L^p (\mathbb{R}^n)} \to 0
\end{align*}
as $j \to \infty$. Since $I_\alpha \phi_j \in C_b(Q)$ for each $j \in \mathbb{N}$, we have by definition of $\mathcal{H}^\beta_\infty$-quasicontinuity that $I_\alpha f$ is $\mathcal{H}^\beta_\infty$-quasicontinuous and thus $I_\alpha f \in L^1(Q; \mathcal{H}^\beta_\infty)$.
\end{proof}

The following lemma presents an analogue of Lemma \ref{Morre1ypreciseestimate} for when $f$ belongs to the weak Lebesgue space $L^{n/\alpha, \infty}(\mathbb{R}^n)$.

\begin{lemma}\label{preciseforweakLnalpha}
       Let $0< \alpha<n$, $0<\beta \leq n$, $Q$ be an open cube in $\mathbb{R}^n$ and $f$ be a real-valued function in $\mathbb{R}^n$. If $\supp{f} \subseteq 2Q$ and $f \in L^{n/ \alpha, \infty } (\mathbb{R}^n)$, then $I_\alpha f \in L^1(Q; \mathcal{H}^\beta_\infty)$ and  for every $1 < p < n/ \alpha$,  there exists a constant $C$ depending on $n$, $\alpha$, $\beta$ such that
\begin{align}\label{weakLnsecond}
     \int_Q | I_\alpha f | \; d \mathcal{H}^\beta_\infty \leq C \ell(Q)^{\beta} \Vert f \Vert_{L^{n/ \alpha, \infty } (\mathbb{R}^n)}.
\end{align}
\end{lemma}
\begin{proof}
Since $ L^{ n/ \alpha, \infty}(\mathbb{R}^n)  \subseteq \mathcal{M}^\alpha_p (\mathbb{R}^n)$ for $ 1 < p < \frac{n}{\alpha}$, we have $I_\alpha f \in L^1 (Q; \mathcal{H}^\beta_\infty)$ for $0<\beta\leq n$ by Lemma \ref{Morreyforplemma}.

Choosing $p \in (1, \frac{n}{\alpha})$ so that $n-\alpha p <\beta$, it then follows again by Lemma \ref{Morreyforplemma} that there exists a constant $C'= C'(n, \alpha ,\beta)$ such that 
\begin{align}\label{quotefrompMorreyextension}
     \int_Q |I_\alpha f| \; d\mathcal{H}^\beta_\infty \leq C' \ell(Q)^{\beta+\alpha -\frac{n}{p}} \Vert f \Vert_{ L^p(\mathbb{R}^n)}
\end{align}
since $f \in  L^{n/ \alpha, \infty}(\mathbb{R}^n) \subseteq \mathcal{M}^\alpha_p(\mathbb{R}^n)$.
Using H\"older's inequality in Lorentz scale, we get
\begin{align}\label{pslorentzestimate}
    \Vert f \Vert_{L^p (\mathbb{R}^n)} \leq C \Vert f \Vert_{L^{n/ \alpha, \infty } (\mathbb{R}^n) } \Vert  \chi_{2Q} \Vert_{L^{s,1} (\mathbb{R}^n)}, 
\end{align}
where $s $ satisfying $\frac{1}{p} = \frac{\alpha}{n} + \frac{1}{s}$.  Since
\begin{align*}
    \Vert \chi_{2Q} \Vert_{L^{s,1} (\mathbb{R}^n)} = C \ell(Q)^{\frac{n}{s}}
\end{align*}
and 
\begin{align*}
    \Vert \chi_{2Q} \Vert_{L^{p'} (\mathbb{R}^n;\mathcal{H}^\beta_\infty)} = C\ell(Q)^{\beta(1- \frac{1}{p})},
\end{align*}
the combination (\ref{quotefrompMorreyextension}) of (\ref{pslorentzestimate}) yields 
\begin{align*}
      \int_Q | I_\alpha f | \; d \mathcal{H}^\beta_\infty \leq C' \ell(Q)^{ \alpha + \beta - \frac{n}{p}} \Vert f \Vert_{L^p(\mathbb{R}^n)}\leq C'' \ell(Q)^{\beta} \Vert f \Vert_{L^{n/ \alpha, \infty } (\mathbb{R}^n)},
\end{align*}
which completes the proof.

\end{proof}
We conclude this section with a lemma that provides an auxiliary estimate for the proofs of Theorem \ref{MainMorreyTheorem} and Theorem \ref{potential_estimate} and it will be used in conjunction with Lemma \ref{preciseforweakLnalpha} and Lemma \ref{Morreyforplemma} in the proofs of Theorem \ref{MainMorreyTheorem} and Theorem \ref{potential_estimate}, respectively.

\begin{lemma}\label{outsidef2inMorrey}
    Let $0<\alpha < n$, $0<\beta \leq n$, $Q$ be an open cube in $\mathbb{R}^n$ with centre $x_0$ and $f$ be a real-valued function with $\supp{f} \subseteq (2Q)^c$. If  $I_\alpha f \in L^1_{loc} (\mathbb{R}^n)$, then $I_\alpha f $ is continuous in Q and there exists $c \in \mathbb{R}$ depending on the cube Q and $C>0$ depending on $n$, $\alpha$, $\beta$  such that 
    \begin{align}\label{Morreyoutsideestimate}
\int_{Q} | I_\alpha f -c| \; d\mathcal{H}^\beta_\infty \leq C \ell(Q)^\beta   \mathcal{M}_\alpha f (x_0).
    \end{align}
 
\end{lemma}

\begin{proof}
    Suppose that $I_\alpha f \in L^1_{loc}(\mathbb{R}^n)$. Then for every $x \in \mathbb{R}^n$, we have $I_\alpha (\tau_x f ) \in L^1_{loc} (\mathbb{R}^n)$, where $\tau_x f $ denotes the translation of $f$ by $x$ defined by $\tau f(y):= f(y+x)$ for $y \in \mathbb{R}^n$. By Lemma \ref{goodtermforlocalintegrabilityofRieszpotential}, we deduce that
\begin{align*}
    \int_{B_N} \frac{f(x+y)}{|y|^{n-\alpha }} \; dy = \int_{B_N} \frac{\tau_x f(y)}{|y|^{n-\alpha }} \; dy < \infty \text{ for any }N>0 
\end{align*}
and thus
\begin{align*}
    \int_{|y| > \frac{\ell(Q)}{2}} \frac{|f(x+y)|}{|y|^{n-\alpha}} \; dy < \infty.
\end{align*}
This implies that for every $x\in Q$,
\begin{align*}
    |I_\alpha f (x)| &\leq   \frac{1}{\gamma(\alpha)}\int_{(2Q)^c} \frac{|f(y)|}{|x-y|^{n-\alpha}} \; dy\\
    &\leq  \frac{1}{\gamma(\alpha)} \int_{|x-y| > \frac{\ell(Q)}{2}} \frac{|f(y)|}{|x-y|^{n-\alpha}} \; dy\\
    &= \frac{1}{\gamma(\alpha)} \int_{|y| > \frac{\ell(Q)}{2}} \frac{|f(x+y)|}{|y|^{n-\alpha}} \; dy\\
    &<\infty.
\end{align*}
Together with the estimate 
\begin{align}\label{meanvalueestimatetwopointsinfirstlemma}
\left|\frac{1}{|x_1 -y|^{n-\alpha}}- \frac{1}{|x_2 - y|^{n-\alpha}}\right| \leq C_{n,\alpha} \frac{|x_1 - x_2|}{|y|^{n-\alpha+1}}
\end{align}
holds for every $x_1,$ $x_2 \in Q$ and $y \in (2Q)^c$, we obtain
\begin{align*}
    |I_\alpha f (x_1) -I_\alpha f (x_2)|& \leq C_{n,\alpha} |x_1 - x_2| \int_{(2Q)^c} \frac{|f(y)|}{|y|^{n-\alpha +1}} \;  dy \\
    &\leq  C_{n,\alpha} |x_1 - x_2| \int_{|y|> \ell(Q)} \frac{|f(y)|}{|y|^{n-\alpha +1}} \;  dy \\
    &\leq C_{n,\alpha} \frac{|x_1 - x_2|}{\ell(Q)} \Vert f \Vert_{\mathcal{M}^{\alpha}_1 (\mathbb{R}^n)} \to 0
\end{align*}
as $|x_1 - x_2| \to 0$.  Thus, $I_\alpha f$ is continuous in $Q$.

Now it remains only to verify (\ref{Morreyoutsideestimate}).  Without loss of generality we take the center of the cube $Q$ to be the origin.  Setting 
\begin{align*}
c:= I_\alpha f (0) = \frac{1}{\gamma(\alpha)}  \int_{(2Q)^c} \frac{1}{|y|^{n-\alpha}} f(y)\;dy  ,
\end{align*} 
we have $|c| < \infty$ by lemma \ref{goodtermforlocalintegrabilityofRieszpotential}. Moreover, inequality (\ref{meanvalueestimatetwopointsinfirstlemma}) yields that 
\begin{align}
 \nonumber \int_{Q} | I_\alpha f(x) -c| \; d\mathcal{H}^\beta_\infty(x)
 &\leq \frac{1}{\gamma(\alpha)}  \int_{Q}   \int_{(2Q)^c} \left|\frac{1}{|x-y|^{n-\alpha}}- \frac{1}{|y|^{n-\alpha}}\right|  |f(y)|\;dy \;d\mathcal{H}^\beta_\infty (x) \\
\nonumber  &\leq C \int_{Q}  |x|   \;  d\mathcal{H}^\beta_\infty (x) \;\int_{(2Q)^c}\frac{|f(y)|}{|y|^{n-\alpha+1}}  \;dy  \\
\nonumber &\leq C l(Q)^{\beta +1} \int_{|y| \geq \ell(Q)}\frac{|f(y)|}{|y|^{n-\alpha+1}} \;dy\\ \nonumber
& \leq C \ell(Q)^\beta \mathcal{M}_{\alpha} f (0),
\end{align}
where we use (\ref{part2forestimateII}) in the forth inequality. This completes the proof.
\end{proof}

\section{Proofs of Main Results}

We now give the proof of Theorem \ref{MainMorreyTheorem}.

\begin{proof}
Let $Q \subseteq \mathbb{R}^n$ be an open cube and 
\begin{align*}
    I_\alpha f (x) & = I_\alpha f_1  (x) + I_\alpha f_2 (x),
\end{align*}
where $f_1 = f \chi_{2Q}$ and $f_2 = f - f_1$.  Then by Lemma \ref{forcompactmorrey}, we have $I_\alpha f_1 \in L^1(Q; \mathcal{H}^{\beta}_ \infty)$ and 
  \begin{align}\label{Main1estimate}
        \int_Q | I_\alpha f_1  | \; d \mathcal{H}^{\beta}_\infty \leq C \ell(Q)^{\beta} \Vert f \Vert_{\mathcal{M}^{\alpha}_p (\mathbb{R}^n)}.
    \end{align}
   Moreover, Lemma \ref{outsidef2inMorrey} yields that $I_\alpha f_2  \in L^1(Q; \mathcal{H}^{\beta}_ \infty)$ and there exists a constant $c \in \mathbb{R}$ such that
\begin{align}\label{Main2estimate}
 \nonumber   \int_{Q} | I_\alpha f_2(x) -c| \; d\mathcal{H}^{\beta}_\infty & \leq  C \ell(Q)^{\beta}  \Vert f \Vert_{
\mathcal{M}_1^{\alpha} (\mathbb{R}^n)}\\
&\leq C \ell(Q)^\beta \Vert f \Vert_{\mathcal{M}^\alpha_p (\mathbb{R}^n)}
.
\end{align}
The combination of (\ref{Main1estimate}) and (\ref{Main2estimate}) then gives
\begin{align*}
        \int_Q | I_\alpha f  | \; d \mathcal{H}^{\beta}_\infty \leq C \ell(Q)^{\beta} \Vert f \Vert_{\mathcal{M}^{\alpha}_p (\mathbb{R}^n)},
    \end{align*}
which completes the proof by taking the supremum over all the cubes $Q$ in $\mathbb{R}^n$.

\end{proof}

We now give the proof of Theorem \ref{potential_estimate}.

\begin{proof}
Let $Q$ be an open cube in $\mathbb{R}^n$. Denote $2Q$ be the cube with the same center and twice the side length and set
\begin{align*}
 I_\alpha f & = I_\alpha f_1 +  I_\alpha f_2,
\end{align*}
where $f_1 := f \chi_{2Q}$ and $f_2:= f - f_1$. Since $\supp{f_1} \subseteq 2Q$ and $f_1 \in L^{n/ \alpha, \infty} (\mathbb{R}^n)$, we have by Lemma \ref{preciseforweakLnalpha} that $ I_\alpha f \in L^1(Q; \mathcal{H}^\beta_\infty)$ and 
\begin{align}\label{f1Hbetaweaknalpha}
     \int_Q | I_\alpha f_1 | \; d \mathcal{H}^\beta_\infty \leq C' \ell(Q)^{ \alpha + \beta -\frac{n}{p}} \Vert f_1 \Vert_{L^{n/\alpha, \infty}(\mathbb{R}^n)}
\end{align}
for any $1 < p < \frac{n}{\alpha}$.

Moreover, since $f_2 \in L^{n/\alpha , \infty} (\mathbb{R}^n) \subseteq \mathcal{M}^{\alpha}_1 (\mathbb{R}^n)$ and $\supp{f_2} \subseteq (2Q)^c$,  an application of Lemma \ref{outsidef2inMorrey} yields that $I_\alpha f_2 \in L^1(Q; \mathcal{H}^\beta_\infty)$ and there exists a constant $c \in \mathbb{R}$ such that 
\begin{align*}
    \int_{Q} | I_\alpha f_2 -c| \; d\mathcal{H}^\beta_\infty &\leq C \ell(Q)^\beta   \|f\|_{\mathcal{M}_1^{\alpha}(\mathbb{R}^n)} \\
    &\leq C \ell(Q)^\beta \Vert f \Vert_{L^{n/\alpha, \infty} (\mathbb{R}^n)}.
\end{align*}
This completes the proof.

\end{proof}

We now give the proof of Theorem \ref{Mainresult}
\begin{proof}
    Suppose that $I_\alpha f \in BMO(\mathbb{R}^n)$. An application of Theorem \ref{Adamsequivalent} and Theorem \ref{MainMorreyTheorem} for $p=1$ yields that for $\beta \in (n-\alpha,n]$,
    \begin{align*}
        \Vert I_\alpha f \Vert_{BMO^\beta(\mathbb{R}^n)} \leq C_1 \Vert f \Vert_{\mathcal{M}^\alpha_1 (\mathbb{R}^n)} \leq C_2 \Vert f \Vert_{BMO (\mathbb{R}^n)}.
    \end{align*}

Now suppose that $I_\alpha f \in BMO^\beta(\mathbb{R}^n)$ for some $\beta \in ( n-\alpha,n]$. It follows by Lemma \ref{nesting} that
\begin{align*}
    \Vert I_\alpha f \Vert_{BMO(\mathbb{R}^n)} \leq C_\beta \Vert I_\alpha f \Vert_{BMO^\beta(\mathbb{R}^n)},
\end{align*}
which completes the proof.

\end{proof}

Finally, we give the proof of Theorem \ref{lotsofequivalent}.

\begin{proof}
    The equivalence of $(i)$ and $(vi)$ is due to Theorem \ref{Adams2015theorem}.  The equivalence of $(i)$ and $(ii)$ is due to Theorem \ref{Mainresult}.  The equivalence of $(ii)$ and $(v)$ is due to Lemma \ref{pseminorms}. Finally, the equivalence of $(ii)$, $(iii)$ and $(iv)$ is due to (\ref{jninequality_content_prime}) and (\ref{betaexpintegral}).
\end{proof}

\end{document}